\documentclass[noinfoline]{imsart}

\usepackage[dvips,letterpaper,textwidth=16cm,vcentering,hcentering]{geometry}
\usepackage{amsmath,amssymb,amsthm,amsfonts,amsbsy}
\usepackage{graphicx}
\usepackage{wrapfig}
\usepackage[normalem]{ulem}
\usepackage[active]{srcltx}
\usepackage{pdfsync}

\usepackage{nicefrac}
 
\usepackage{enumerate}
\usepackage{appendix}
\usepackage{multirow}
\usepackage{afterpage}
\usepackage{cite}
\usepackage{fancyhdr}
\usepackage[crop=pdfcrop]{pstool}
\usepackage{subfig}
\usepackage{tocvsec2}
\usepackage{xr,xr-hyper}

\usepackage{calrsfs}

\usepackage{hyperref}

\providecommand{\url}[1]{#1}

\renewcommand{\le}{\leqslant}
\renewcommand{\ge}{\geqslant}

\renewcommand{\P}{\operatorname{\mathsf{P}}} 
\newcommand{\E}{\operatorname{\mathsf{E}}}

\newcommand{\cc}{C^2_\mathsf{conv}}

\newcommand{\R}{\mathbb{R}}

\newcommand{\al}{\alpha}
\newcommand{\be}{\beta}

\newcommand{\De}{\Delta}

\newcommand{\vp}{\varepsilon}
\newcommand{\vpp}{\boldsymbol{\vp}}
\newcommand{\xii}{\boldsymbol{\xi}}
\newcommand{\all}{\boldsymbol{\al}}

\newcommand{\XX}%{\mathfrak{X}}
{\R^d}

\newtheorem{theoremm}{Theorem}
\newtheorem{lemmaa}%[theorem]
{Lemma}
\newtheorem{cor}%[theorem]
{Corollary}

{Lemma}

\theoremstyle{remark}

\begin{document} 

\begin{frontmatter}

\title{On a multi%IP03-22-16variate
dimensional spherically invariant extension of the Rademacher--Gaussian comparison}
\runtitle{Spherically invariant extension of the Rademacher--Gaussian comparison}

\begin{aug}
 \author{\fnms{Iosif}  \snm{Pinelis}\corref{}\ead[label=e1]{ipinelis@mtu.edu}}
 \affiliation{Michigan Technological University}
 \address{Department of Mathematical Sciences\\Michigan Technological University\\Hough\-ton, Michigan 49931\\\printead{e1}}
%IP03.15.16
\runauthor{I.\ Pinelis} 
\end{aug} 

\begin{abstract} %IP03-22-16 
It is shown that
\begin{equation*}
	\mathsf{P}(\|a_1U_1+\dots+a_nU_n\|>u)\le c\,\mathsf{P}(a\|Z_d\|>u)  
\end{equation*} 
for all real $u$, where $U_1,\dots,U_n$ are independent random vectors uniformly distributed on the unit sphere in $\mathbb{R}^d$, $a_1,\dots,a_n$ are any real numbers, $a:=\sqrt{(a_1^2+\dots+a_n^2)/d}$, $Z_d$ is a standard normal random vector in $\mathbb{R}^d$, and $c=2e^3/9=4.46\dots$. 
This constant factor is about $89$ times as small as 
the one in a recent result by Nayar and Tkocz, who proved, %IP03.15.16
by a different method, a corresponding conjecture by Oleszkiewicz. 
As an immediate application, a corresponding upper bound on the tail probabilities for the norm of the sum of arbitrary independent spherically invariant random vectors is given.
\end{abstract}

\setattribute{keyword}{AMS}{AMS 2010 subject classifications:}


% 	60F05   	Central limit and other weak theorems
%  	62F12   	Asymptotic properties of estimators

% 	62E17   	Approximations to distributions (nonasymptotic)
%		62E20   	Asymptotic distribution theory

%	62F10   	Point estimation


\end{frontmatter}

%\tableofcontents %RMejs

%\section{Introduction}
%\label{intro}

%\section{Proofs}\label{sec:proofs}

%IP03-22-16

Usually, at the heart of any good limit theorem is at least one good inequality. This should become clear if one recalls the definition of the limit and the fact that a neighborhood of a point in a specific topology is usually defined in terms of inequalities. A limit theorem can be very illuminating. However, it only describes the behavior of a function near a given point (possibly at infinity), whereas a corresponding inequality would cover an entire range. 

Also, the nature of limit theorems is more qualitative, whereas that of inequalities is more quantitative. E.g., a central limit theorem would state that a certain distribution is close to normality; such a statement by itself is qualitative, as it does not specify the degree of closeness under specific conditions. In contrast, a corresponding Berry--Esseen-type inequality can provide such quantitative specifics. 

This is why good inequalities are important. A good inequality would be, not only broadly enough applicable, but also precise enough. Indeed, only such results have a chance to be adequately used in real-world applications. 
Such an understanding of the role of good and, in particular, best possible bounds goes back at least to Chebyshev; cf.\ the theory of Tchebycheff systems \cite{karlin-studden,krein-nudelman} developed to provide optimal solutions to a broad class of such problems. 
These ideas were further developed by a large number of authors, including Bernstein \cite{bernstein}, Bennett \cite{bennett}, and Hoeffding \cite{hoeff-extr,hoeff63}. 
In particular, Bennett \cite{bennett} exerted a considerable effort on comparing various bounds on tail probabilities in various ranges. Quoting Bennett \cite{bennett}:

\begin{quote}
Much work has been carried out on the asymptotic form of the distribution
of such sums [of independent random variables] when the number of component random variables is large and/or
when the component variables have identical distributions. The majority of
this work, while being suitable for the determination of the asymptotic distribution
of sums of random variables, does not provide estimates of the accuracy
of such asymptotic distributions when applied to the summation of finite numbers
of components. [...] Yet, for most practical problems,
precisely this distribution function is required.
\end{quote}

In this note, we shall present an upper bound on a tail probability that is about $89$ times as small as the corresponding bound recently obtained in \cite{nayar-tkocz}. 

To provide a relevant context, let us begin by introducing the class $\cc$ of all even twice differentiable functions $h\colon\R\to\R$ whose second derivative $h''$ is convex. Let $\vp,\vp_1,\dots,\vp_n$ be independent Rademacher random variables (r.v.'s), and let $\xi_1,\dots,\xi_n$ be any independent symmetric r.v.'s with $\E\xi_i^2=1$ for all $i$. 

Take any natural $d$. For any vectors $x$ and $y$ in $\R^d$, let, as usual, $x\cdot y$ denote the standard inner product of $x$ and $y$, and then let $\|x\|:=\sqrt{x\cdot x}$.   

Theorem~2.3 in \cite{T2} states that $\E h\big(\sqrt{\vpp A\vpp^T}\big)\le\E h\big(\sqrt{\xii A\xii^T}\big)$ for any $h\in\cc$ and any nonnegative definite $n\times n$ matrix $A\in\R^{n\times n}$, where $\vpp:=[\vp_1,\dots,\vp_n]$ and $\xii:=[\xi_1,\dots,\xi_n]$. 
This 
can be restated as the following generalized moment comparison: 
\begin{equation}\label{eq:pin-compar}
	\E h(\|\vp_1x_1+\dots+\vp_nx_n\|)\le\E h(\|\xi_1x_1+\dots+\xi_nx_n\|)
\end{equation}
for any $h\in\cc$ and any (nonrandom) vectors $x_1,\dots,x_n$ in $\R^d$; indeed, any nonnegative definite matrix $A\in\R^{n\times n}$ is the Gram matrix of some vectors $x_1,\dots,x_n$ in $\R^d$ for some natural $d$, and then $\|\al_1x_1+\dots+\al_nx_n\|=\sqrt{\all A\all^T}$ for any $\all:=[\al_1,\dots,\al_n]\in\R^{1\times n}$. 
From the comparison \eqref{eq:pin-compar} of generalized moments of the r.v.'s $\|\vp_1x_1+\dots+\vp_nx_n\|$ and $\|\xi_1x_1+\dots+\xi_nx_n\|$, a tail comparison was extracted (\cite[Theorem~2.4]{T2}), an equivalent form of which is the inequality 
\begin{equation}\label{eq:ortho}
	\P(\|\vp_1x_1+\dots+\vp_nx_n\|>u)<c\P(\|Z_r\|>u)
\end{equation}
for all real $u$, where $x_1,\dots,x_n$ are any (nonrandom) vectors in $\R^d$ whose Gram matrix is an orthoprojector of rank $r$, $Z_r$ is a standard normal random vector in $\R^r$, and 
\begin{equation}\label{eq:c=c_3}
	c=c_3:=2e^3/9=4.46\dots. 
\end{equation}
A special case of \eqref{eq:ortho} is the inequality 
\begin{equation}\label{eq:d=1}
	\P(|\vp_1a_1+\dots+\vp_na_n|>u)\le c\P(|Z_1|>u)
\end{equation}
for all real $u$, where $a_1,\dots,a_n$ are any real numbers such that 
\begin{equation*}
	a_1^2+\dots+a_n^2=1. 
\end{equation*}
The quoted results generalize and %IP03.17.16 /or 
refine results of \cite{eaton1,eaton2}. In turn, they were further developed in \cite{pin98,pin99}.  

A simple inductive argument, which was direct rather than based on a generalized moment comparison, was offered in \cite{BGH}, where \eqref{eq:d=1} was proved with $c\approx12.01$.  
Based in part on that inductive argument in \cite{BGH}, the constant $c$ in \eqref{eq:d=1} was improved to $\approx1.01c_*$ in \cite{pin-towards} and then to $c_*$ in \cite{bent-dzin_publ}, where $c_*:=\P(|\vp_1+\vp_2|\ge2)/\P(|Z_1|\ge\sqrt2)=3.17\dots$, so that $c_*$ is the best possible value of $c$ in \eqref{eq:d=1}. 

In \cite{baern-culver}, another kind of multidimensional generalized moment comparison was obtained. A continuous function $f\colon\R^d\to\R$ is called bisubharmonic if the (Sobolev--Schwartz) distribution $\De^2f$ is a nonnegative Radon measure on $\R^d$, where $\De$ is the Laplace operator on $\R^d$. 
By \cite[Theorem~3]{baern-culver}, for any continuous function $f\colon\R^d\to\R$ one has 
\begin{equation}\label{eq:iff}
\text{$f$ is bisubharmonic if and only if $\E f(y+U\sqrt t)$ is convex in $t\in(0,\infty)$ for each $y\in\R^d$,} 	
\end{equation}
where $U$ is a random vector uniformly distributed on the unit sphere $S^{d-1}$ in $\R^d$. 

Let $U_1,\dots,U_n$ be independent copies of $U$. 
Theorem~1 in \cite{baern-culver} states that 
\begin{equation}\label{eq:BC-compar}
	\E f(a_1U_1+\dots+a_nU_n)\le\E f(b_1U_1+\dots+b_nU_n), 
\end{equation}
where $f$ is a bisubharmonic function and $a_1,\dots,a_n,b_1,\dots,b_n$ are real numbers such that the %IP03.15.16 $d$
$n$-tuple $(b_1^2,\dots,b_n^2)$ is majorized by $(a_1^2,\dots,a_n^2)$ in the sense of the Schur majorization (see e.g.\ \cite{marsh-ol}). 

One may note that, whereas in \eqref{eq:pin-compar} each of the random summands $\vp_1x_1,\dots,\vp_nx_n,%\break 
\xi_1x_1,\dots,\xi_nx_n$ is distributed on a straight line through the origin, each of the random summands  $a_1U_1,\dots,a_nU_n,b_1U_1,\dots,b_nU_n$ in \eqref{eq:BC-compar} is uniformly distributed on a sphere centered at the origin. 

Since the distributions of the random vectors $a_1U_1+\dots+a_nU_n$ and $b_1U_1+\dots+b_nU_n$ are clearly spherically invariant, without loss of generality one may assume that the function $f$ in \eqref{eq:BC-compar} is spherically invariant as well, that is, $f(x)$ depends on $x\in\R^d$ only through $\|x\|$. 
If $f$ is indeed a spherically invariant bisubharmonic function, it then follows from \eqref{eq:BC-compar} and \cite[formulas~(1.2), (1.3)]{baern-culver} that 
\begin{equation}\label{eq:BC-compar,Z}
	\E f(a_1U_1+\dots+a_nU_n)\le\E f(aZ_d), 
\end{equation}
where 
%%IP03-15-16\begin{equation}
%	a:=\sqrt{(a_1^2+\dots+a_n^2)/d};  
%\end{equation}
\begin{equation}\label{eq:a}
	a:=\sqrt{(a_1^2+\dots+a_n^2)/d};  
\end{equation}
cf.\ \cite[Corollary~1]{baern-culver}. 

Let $\cc(H)$ denote the class of all spherically invariant twice differentiable functions $f$ from a Hilbert space $H$ to $\R$ whose second derivative $f''$ is convex in the sense that the function $H\ni x\mapsto f''(x;y,y)$ is convex for each $y\in H$, where $f''(x;y,y)$ is the value of the second derivative of the function $\R\ni t\mapsto f(x+ty)$ at $t=0$.  
The class $\cc(H)$ was characterized in \cite{spher}, with some applications. Clearly, $\cc(\R)$ coincides with the class $\cc$ defined in the beginning of this note. 

K.\ Oleszkiewicz conjectured \cite{nayar-tkocz} that 
\begin{equation}\label{eq:KO-compar}
	\P(\|a_1U_1+\dots+a_nU_n\|>u)\le c\P(a\|Z_d\|>u)  
\end{equation}
for some universal constant $c$ and all real $u$, where $a_1,\dots,a_n,a,U_1,\dots,U_n,Z_d$ are as before; clearly, \eqref{eq:KO-compar} is a generalization of \eqref{eq:d=1}. 
This conjecture was proved in \cite{nayar-tkocz} with $c=397$ %IP03.15.16 
based, in part, on the idea from \cite{BGH}. 

Using inequality (2.6) in \cite{T2}, one can improve the lower bound $1/397$ in \cite[Lemma 1]{nayar-tkocz} to $1/e^2$ and thus improve the constant $c$ in \eqref{eq:KO-compar} from $397$ to $e^2=7.38\dots$. 
Indeed, let, as usual, $\Phi$ denote the standard normal distribution function. 
Then, 
by
inequality (2.6) in \cite{T2},
$g(d):=\P(\|Z_d\|\ge\sqrt{d+2}\,)>1-\Phi\big((\sqrt{d+2}-\sqrt{d-1}\,)\sqrt2\,\big)=:q(d)$,
which latter is clearly increasing in $d$, with $q(4)>1/e^2$, whence
$g(d)>1/e^2$ for $d=4,5,...$, whereas $g(2)=1/e^2<g(3)$. So,
$\P(\|Z_d\|\ge\sqrt{d+2}\,)=g(d)\ge1/e^2$ for $d=2,3,...$.
Similarly, $\P(\|Z_d\|\ge\sqrt d\,)\ge1/e$ for $d=2,3,...$ (but a lower bound on $\P(\|Z_d\|\ge\sqrt d\,)$ is not really needed in the proof of the main result in \cite{nayar-tkocz}).

The aim of this note is to point out that, %IP03.15.16
based on the generalized moment comparison \eqref{eq:BC-compar,Z} and results in \cite{T2,pin98}, one can further improve the constant $c$ in \eqref{eq:KO-compar}: 

%\newpage

\begin{theoremm}\label{th:}
Inequality \eqref{eq:KO-compar} holds (for all real $u$) with $c$ as in \eqref{eq:c=c_3}. 
The strict version of \eqref{eq:KO-compar}, again with $c$ as in \eqref{eq:c=c_3}, also holds. 
\end{theoremm}

Our method is quite different from that of \cite{nayar-tkocz}. 
In view of \eqref{eq:BC-compar,Z}, Theorem~\ref{th:} is an immediate corollary of the following two lemmas. 

\begin{lemmaa}\label{lem:1}
For any function $h\in\cc$, the function $f\colon\R^d\to\R$ defined by the formula $f(x):=h(\|x\|)$ for $x\in\R^d$ is a %IP03.15.16 a 
spherically invariant bisubharmonic function. 
\end{lemmaa}

\begin{lemmaa}\label{lem:2}
Let $\xi$ be any nonnegative r.v.\ such that 
\begin{equation}\label{eq:E h(xi)<}
	\E h(\xi)\le\E h(\|Z_d\|)\quad\text{for all}\quad h\in\cc. 
\end{equation}
Then %\break 
\begin{equation}\label{eq:<c_3...}
\P(\xi>u)<c_3\P(\|Z_d\|>u) 	
\end{equation}
for all real $u$, with $c_3$ defined in \eqref{eq:c=c_3}. 
\end{lemmaa}

\begin{proof}[Proof of Lemma~\ref{lem:1}]
%Let us show that $f$ is bisubharmonic. 
%By \cite[Theorem~3]{baern-culver}, it suffices to show that the function $(0,\infty)\ni t\mapsto\E f(y+U\sqrt t)$ is convex for each $y\in\R^d$, where $U:=U_1$. 
Let $U$ be as in \eqref{eq:iff} and then let $\vp$ be a Rademacher r.v.\ independent of $U$. 
For all $t\in(0,\infty)$ and $y\in\R^d$
\begin{equation}\label{eq:lem1}
	\E f(y+U\sqrt t)=\E f(y+\vp U\sqrt t)=\E h(\|y+\vp U\sqrt t\|)=\E\E_U g_{b_U,h}(\be_U+\vp\sqrt t), 
\end{equation}
where %$U:=U_1$ and $\vp$ is a Rademacher r.v.\ independent of $U$, 
$\E_U$ denotes the conditional expectation given $U$, $g_{b,h}(u):=h\big(\sqrt{u^2+b}\,\big)$ for $b\in[0,\infty)$ and $u\in\R$, $\be_U:=y\cdot U$, and $b_U:=\|y\|^2-(y\cdot U)^2\ge0$, so that the r.v.\ $\vp$ is independent of the pair $(b_U,\be_U)$, which latter is a function of $U$. 
%Take any $b\in[0,\infty)$ and $\be\in\R$. 
By \cite[Lemma~3.1]{T2}, $g_{b,h}\in\cc$ for each $b\in[0,\infty)$. 
Hence, by \cite[Lemma~3.1]{utev-extr} or \cite[Proposition~A.1]{T2}, $\E_U g_{b_U,h}(\be_U+\vp\sqrt t)$ is convex in $t\in(0,\infty)$.  
So, in view of \eqref{eq:lem1}, $\E f(y+U\sqrt t)$ is convex in $t\in(0,\infty)$. 
Now %\cite[Theorem~3]{baern-culver} implies 
it follows by \eqref{eq:iff} that the function $f$ is indeed bisubharmonic. 
That $f$ is spherically invariant is trivial.  
\end{proof}

\begin{proof}[Proof of Lemma~\ref{lem:2}]
Taken almost verbatim, the proof of Theorem~2.4 in \cite{T2} (based on Theorem~2.3 in \cite{T2}) can also serve as a proof of Lemma~\ref{lem:2}. Indeed, no properties of the r.v.\ $\vp\Pi\vp^T$ were used in the proof of \cite[Theorem~2.4]{T2} except that this nonnegative r.v.\ satisfies the inequality in \cite[Theorem~2.3]{T2} with $A=\Pi$ and $\xi=Z_n$, which can then be written as \eqref{eq:E h(xi)<} with $\xi=\sqrt{\vp\Pi\vp^T}$ and $d$ equal the rank of $\Pi$. (Note here a typo in \cite{T2}: in place of ``Theorem~2.3'' in line 7- on page 363 there, it should be ``Theorem~2.4''.) 

Instead of following the entire proof of \cite[Theorem~2.4]{T2}, one can alternatively reason as follows. 
Let $\xi$ be any nonnegative r.v.\ such that \eqref{eq:E h(xi)<} holds. Then \cite[Lemma~3.5]{T2} holds with $\xi^2$ in place of $\vpp\Pi\vpp^T$. So, in view of \cite[formula~(3.11)]{T2} and \cite[formula~(22) in Theorem~3.11]{pin98}, inequality \eqref{eq:<c_3...} holds for $u\ge\mu_r$, with $r:=d$ and $\mu_r$ defined on page~362 in \cite{T2}. 
The cases $r^{1/2}\le u\le\mu_r$ and $0\le u\le r^{1/2}$ are considered as %IP03.15.16 
was done at the end of the proof of \cite[Lemma~3.6]{T2}, starting at the middle of page~365 in \cite{T2}. The case $u<0$ is trivial. 
\end{proof}

An immediate application of Theorem~\ref{th:} is  

\begin{cor}\label{cor:}
Let $X_1,\dots,X_n$ be any independent spherically invariant random vectors in $\R^d$, which are also independent of the Gaussian random vector $Z_d$. Then 
\begin{equation}\label{eq:sph-inv}
	\P(\|X_1+\dots+X_n\|>u)<\frac{2e^3}9\,\P\big(\sqrt{\|X_1\|^2+\dots+\|X_n\|^2}\;\|Z_d\|>u\big) 
\end{equation}
for all real $u$. 
\end{cor}

This corollary follows from Theorem~\ref{th:} by the conditioning on $\|X_1\|,\dots,\|X_n\|$, because for each $i=1,\dots,n$ the conditional distribution of the spherically invariant random vector $X_i$ given $\|X_i\|=a_i$ is the distribution of $a_iU_i$. 

In the case when the independent spherically invariant random vectors $X_1,\dots,X_n$ are bounded almost surely by positive real numbers $b_1,\dots,b_n$, respectively, one can obviously replace $\sqrt{\|X_1\|^2+\dots+\|X_n\|^2}$ in the bound in \eqref{eq:sph-inv} by $\sqrt{b_1^2+\dots+b_n^2}$. The resulting bound, but with the constant factor $397$ in place of $\frac{2e^3}9=4.46\dots$, was obtained in \cite{nayar-tkocz}. 

Similarly to the extension \eqref{eq:sph-inv} of inequality \eqref{eq:KO-compar}, one can extend \eqref{eq:BC-compar,Z} as follows: 
\begin{equation}\label{eq:sph-inv,E}
	\E f(X_1+\dots+X_n)\le\E f\big(\sqrt{\|X_1\|^2+\dots+\|X_n\|^2}\;Z_d\big) 
\end{equation}
for any spherically invariant bisubharmonic function $f$, 
where $X_1,\dots,X_n$ are as in Corollary~\ref{cor:}. 

A related result was obtained in \cite{kwapien_best-khin-rot-inv}: if $X_1,\dots,X_n$ are independent identically distributed spherically invariant random vectors in $\R^d$ such that $\E h(\|X_1\|^2)\le\E h(\|Z_d\|^2)$ for all nonnegative convex functions $h\colon\R\to\R$, then 
\begin{equation}\label{eq:kwapien}
	\E\|a_1X_1+\dots+a_nX_n\|^p\le\E\|aZ_d\sqrt d\|^p 
\end{equation}
for real $p\ge3$, where $a_1,\dots,a_n,a$ are as in \eqref{eq:BC-compar,Z}--\eqref{eq:a}.

\bibliographystyle{imsart-number}
%\bibliography{citations.nodoi}
\bibliography{C:/Users/ipinelis/Dropbox/mtu/bib_files/citations12.13.12}

\def\cprime{$'$} \def\polhk#1{\setbox0=\hbox{#1}{\ooalign{\hidewidth
  \lower1.5ex\hbox{`}\hidewidth\crcr\unhbox0}}}
  \def\polhk#1{\setbox0=\hbox{#1}{\ooalign{\hidewidth
  \lower1.5ex\hbox{`}\hidewidth\crcr\unhbox0}}}
  \def\polhk#1{\setbox0=\hbox{#1}{\ooalign{\hidewidth
  \lower1.5ex\hbox{`}\hidewidth\crcr\unhbox0}}} \def\cprime{$'$}
  \def\polhk#1{\setbox0=\hbox{#1}{\ooalign{\hidewidth
  \lower1.5ex\hbox{`}\hidewidth\crcr\unhbox0}}}
  \def\polhk#1{\setbox0=\hbox{#1}{\ooalign{\hidewidth
  \lower1.5ex\hbox{`}\hidewidth\crcr\unhbox0}}} \def\cprime{$'$}
  \def\cprime{$'$}
\begin{thebibliography}{20}
% BibTex style file: imsart-number.bst, 2013-01-28
% Default style options (sort=1,type=number).
% Used options (sort=1,type=number).

\bibitem{baern-culver}
\begin{barticle}[author]
\bauthor{\bsnm{Baernstein},~\bfnm{Albert}\binits{A.} \bsuffix{II}} \AND
  \bauthor{\bsnm{Culverhouse},~\bfnm{Robert~C.}\binits{R.~C.}}
(\byear{2002}).
\btitle{Majorization of sequences, sharp vector {K}hinchin inequalities, and
  bisubharmonic functions}.
\bjournal{Studia Math.}
\bvolume{152}
\bpages{231--248}.
\bdoi{10.4064/sm152-3-3}
\bmrnumber{1916226}
\end{barticle}
\endbibitem

\bibitem{bennett}
\begin{barticle}[author]
\bauthor{\bsnm{Bennett},~\bfnm{George}\binits{G.}}
(\byear{1962}).
\btitle{Probability Inequalities for the Sum of Independent Random Variables}.
\bjournal{J. Amer. Statist. Assoc.}
\bvolume{57}
\bpages{33--45}.
\end{barticle}
\endbibitem

\bibitem{bent-dzin_publ}
\begin{barticle}[author]
\bauthor{\bsnm{Bentkus},~\bfnm{Vidmantas~Kastytis}\binits{V.~K.}} \AND
  \bauthor{\bsnm{Dzindzalieta},~\bfnm{Dainius}\binits{D.}}
(\byear{2015}).
\btitle{A tight {G}aussian bound for weighted sums of {R}ademacher random
  variables}.
\bjournal{Bernoulli}
\bvolume{21}
\bpages{1231--1237}.
\bdoi{10.3150/14-BEJ603}
\bmrnumber{3338662}
\end{barticle}
\endbibitem

\bibitem{bernstein}
\begin{barticle}[author]
\bauthor{\bsnm{Bernstein},~\bfnm{S.}\binits{S.}}
(\byear{1924}).
\btitle{Sur une modification de l'in\'equalit\'e de {T}chebichef}.
\bjournal{Ann. Sc. Instit. Sav. Ukraine, Sect. Math.}
\bvolume{I}
\bpages{38--49}.
\end{barticle}
\endbibitem

\bibitem{BGH}
\begin{barticle}[author]
\bauthor{\bsnm{Bobkov},~\bfnm{Sergey~G.}\binits{S.~G.}},
  \bauthor{\bsnm{G{\"o}tze},~\bfnm{Friedrich}\binits{F.}} \AND
  \bauthor{\bsnm{Houdr{\'e}},~\bfnm{Christian}\binits{C.}}
(\byear{2001}).
\btitle{On {G}aussian and {B}ernoulli covariance representations}.
\bjournal{Bernoulli}
\bvolume{7}
\bpages{439--451}.
\bdoi{10.2307/3318495}
\bmrnumber{MR1836739 (2002g:60038)}
\end{barticle}
\endbibitem

\bibitem{eaton1}
\begin{barticle}[author]
\bauthor{\bsnm{Eaton},~\bfnm{Morris~L.}\binits{M.~L.}}
(\byear{1970}).
\btitle{A note on symmetric {B}ernoulli random variables}.
\bjournal{Ann. Math. Statist.}
\bvolume{41}
\bpages{1223--1226}.
\bmrnumber{MR0268930 (42 \#\#3827)}
\end{barticle}
\endbibitem

\bibitem{eaton2}
\begin{barticle}[author]
\bauthor{\bsnm{Eaton},~\bfnm{Morris~L.}\binits{M.~L.}}
(\byear{1974}).
\btitle{A probability inequality for linear combinations of bounded random
  variables}.
\bjournal{Ann. Statist.}
\bvolume{2}
\bpages{609--613}.
\end{barticle}
\endbibitem

\bibitem{hoeff-extr}
\begin{barticle}[author]
\bauthor{\bsnm{Hoeffding},~\bfnm{Wassily}\binits{W.}}
(\byear{1955}).
\btitle{The extrema of the expected value of a function of independent random
  variables}.
\bjournal{Ann. Math. Statist.}
\bvolume{26}
\bpages{268--275}.
\bmrnumber{MR0070087 (16,1128g)}
\end{barticle}
\endbibitem

\bibitem{hoeff63}
\begin{barticle}[author]
\bauthor{\bsnm{Hoeffding},~\bfnm{Wassily}\binits{W.}}
(\byear{1963}).
\btitle{Probability inequalities for sums of bounded random variables}.
\bjournal{J. Amer. Statist. Assoc.}
\bvolume{58}
\bpages{13--30}.
\bmrnumber{MR0144363 (26 \#\#1908)}
\end{barticle}
\endbibitem

\bibitem{karlin-studden}
\begin{bbook}[author]
\bauthor{\bsnm{Karlin},~\bfnm{Samuel}\binits{S.}} \AND
  \bauthor{\bsnm{Studden},~\bfnm{William~J.}\binits{W.~J.}}
(\byear{1966}).
\btitle{Tchebycheff systems: {W}ith applications in analysis and statistics}.
\bseries{Pure and Applied Mathematics, Vol. XV}.
\bpublisher{Interscience Publishers John Wiley \& Sons, New
  York-London-Sydney}.
\bmrnumber{MR0204922 (34 \#\#4757)}
\end{bbook}
\endbibitem

\bibitem{kwapien_best-khin-rot-inv}
\begin{barticle}[author]
\bauthor{\bsnm{K{\"o}nig},~\bfnm{H.}\binits{H.}} \AND
  \bauthor{\bsnm{Kwapie{\'n}},~\bfnm{S.}\binits{S.}}
(\byear{2001}).
\btitle{Best {K}hintchine type inequalities for sums of independent,
  rotationally invariant random vectors}.
\bjournal{Positivity}
\bvolume{5}
\bpages{115--152}.
\bdoi{10.1023/A:1011434208929}
\bmrnumber{1825172 (2002a:60023)}
\end{barticle}
\endbibitem

\bibitem{krein-nudelman}
\begin{bbook}[author]
\bauthor{\bsnm{Kre{\u\i}n},~\bfnm{M.~G.}\binits{M.~G.}} \AND
  \bauthor{\bsnm{Nudel{\cprime}man},~\bfnm{A.~A.}\binits{A.~A.}}
(\byear{1977}).
\btitle{The {M}arkov moment problem and extremal problems}.
\bpublisher{American Mathematical Society}, \baddress{Providence, R.I.}
\bnote{Ideas and problems of P. L. {\v{C}}eby{\v{s}}ev and A. A. Markov and
  their further development, Translated from the Russian by D. Louvish,
  Translations of Mathematical Monographs, Vol. 50}.
\bmrnumber{MR0458081 (56 \#\#16284)}
\end{bbook}
\endbibitem

\bibitem{marsh-ol}
\begin{bbook}[author]
\bauthor{\bsnm{Marshall},~\bfnm{Albert~W.}\binits{A.~W.}} \AND
  \bauthor{\bsnm{Olkin},~\bfnm{Ingram}\binits{I.}}
(\byear{1979}).
\btitle{Inequalities: theory of majorization and its applications}.
\bseries{Mathematics in Science and Engineering}
\bvolume{143}.
\bpublisher{Academic Press Inc. [Harcourt Brace Jovanovich Publishers]},
  \baddress{New York}.
\bmrnumber{MR552278 (81b:00002)}
\end{bbook}
\endbibitem

\bibitem{nayar-tkocz}
\begin{barticle}[author]
\bauthor{\bsnm{Nayar},~\bfnm{Piotr}\binits{P.}} \AND
  \bauthor{\bsnm{Tkocz},~\bfnm{Tomasz}\binits{T.}}
(\byear{2016}).
\btitle{{A multidimensional analogue of the {R}ademacher-{G}aussian tail
  comparison}}.
\bnote{arXiv:1602.07995 [math.PR], \url{http://arxiv.org/abs/1602.07995}}.
\arxiv{arXiv:1602.07995 [math.PR]}
\end{barticle}
\endbibitem

\bibitem{T2}
\begin{barticle}[author]
\bauthor{\bsnm{Pinelis},~\bfnm{Iosif}\binits{I.}}
(\byear{1994}).
\btitle{Extremal probabilistic problems and {H}otelling's {$T^2$} test under a
  symmetry condition}.
\bjournal{Ann. Statist.}
\bvolume{22}
\bpages{357--368}.
\bdoi{10.1214/aos/1176325373}
\bmrnumber{MR1272088 (95m:62115)}
\end{barticle}
\endbibitem

\bibitem{pin98}
\begin{bincollection}[author]
\bauthor{\bsnm{Pinelis},~\bfnm{Iosif}\binits{I.}}
(\byear{1998}).
\btitle{Optimal tail comparison based on comparison of moments}.
In \bbooktitle{High dimensional probability ({O}berwolfach, 1996)}.
\bseries{Progr. Probab.}
\bvolume{43}
\bpages{297--314}.
\bpublisher{Birkh\"auser}, \baddress{Basel}.
\bmrnumber{MR1652335 (2000a:60026)}
\end{bincollection}
\endbibitem

\bibitem{pin99}
\begin{bincollection}[author]
\bauthor{\bsnm{Pinelis},~\bfnm{Iosif}\binits{I.}}
(\byear{1999}).
\btitle{Fractional sums and integrals of {$r$}-concave tails and applications
  to comparison probability inequalities}.
In \bbooktitle{Advances in stochastic inequalities ({A}tlanta, {GA}, 1997)}.
\bseries{Contemp. Math.}
\bvolume{234}
\bpages{149--168}.
\bpublisher{Amer. Math. Soc.}, \baddress{Providence, RI}.
\bmrnumber{MR1694770 (2000k:60027)}
\end{bincollection}
\endbibitem

\bibitem{spher}
\begin{barticle}[author]
\bauthor{\bsnm{Pinelis},~\bfnm{Iosif}\binits{I.}}
(\byear{2002}).
\btitle{Spherically symmetric functions with a convex second derivative and
  applications to extremal probabilistic problems}.
\bjournal{Math. Inequal. Appl.}
\bvolume{5}
\bpages{7--26}.
\bmrnumber{MR1880267 (2003e:60039)}
\end{barticle}
\endbibitem

\bibitem{pin-towards}
\begin{barticle}[author]
\bauthor{\bsnm{Pinelis},~\bfnm{Iosif}\binits{I.}}
(\byear{2007}).
\btitle{Toward the best constant factor for the {R}ademacher-{G}aussian tail
  comparison}.
\bjournal{ESAIM Probab. Stat.}
\bvolume{11}
\bpages{412--426}.
\bdoi{10.1051/ps:2007027}
\bmrnumber{MR2339301 (2008e:60045)}
\end{barticle}
\endbibitem

\bibitem{utev-extr}
\begin{bincollection}[author]
\bauthor{\bsnm{Utev},~\bfnm{S.~A.}\binits{S.~A.}}
(\byear{1985}).
\btitle{Extremal problems in moment inequalities}.
In \bbooktitle{Limit theorems of probability theory}.
\bseries{Trudy Inst. Mat.}
\bvolume{5}
\bpages{56--75, 175}.
\bpublisher{``Nauka'' Sibirsk. Otdel.}, \baddress{Novosibirsk}.
\bmrnumber{MR821753 (87d:60021)}
\end{bincollection}
\endbibitem

\end{thebibliography}

\end{document}